\documentclass{amsart}
\usepackage{amssymb,amsthm,amsmath,epsfig,latexsym,xypic}
\usepackage{calc}
\usepackage{latexsym}
\usepackage{amscd,amssymb,subfigure,hyperref}
\usepackage[arrow,matrix,graph,frame,poly,arc,tips]{xy}

\begin{document}

% \mmbox enables macros to survive outside of $ ... $
\newcommand{\mmbox}[1]{\mbox{${#1}$}}
\newcommand{\affine}[1]{\mmbox{{\mathbb A}^{#1}}}
\newcommand{\Ann}[1]{\mmbox{{\rm Ann}({#1})}}
\newcommand{\caps}[3]{\mmbox{{#1}_{#2} \cap \ldots \cap {#1}_{#3}}}
\newcommand{\N}{{\mathbb N}}
\newcommand{\Z}{{\mathbb Z}}
\newcommand{\R}{{\mathbb R}}
\newcommand{\A}{{\mathcal A}}
\newcommand{\B}{{\mathcal B}}
\newcommand{\C}{{\mathbb K}}
\newcommand{\PP}{{\mathbb P}}
\newcommand{\cO}{{\mathcal O}}
\newcommand{\Tor}{\mathop{\rm Tor}\nolimits}
\newcommand{\ot}{\mathop{\rm OT}\nolimits}
\newcommand{\ao}{\mathop{\rm AOT}\nolimits}
\newcommand{\Ext}{\mathop{\rm Ext}\nolimits}
\newcommand{\Hom}{\mathop{\rm Hom}\nolimits}
\newcommand{\Cl}{\mathop{\rm Cl}\nolimits}
\newcommand{\im}{\mathop{\rm Im}\nolimits}
\newcommand{\rank}{\mathop{\rm rank}\nolimits}
\newcommand{\codim}{\mathop{\rm codim}\nolimits}
\newcommand{\supp}{\mathop{\rm supp}\nolimits}
\newcommand{\CB}{Cayley-Bacharach}
\newcommand{\HF}{\mathrm{HF}}
\newcommand{\HP}{\mathrm{HP}}
\newcommand{\coker}{\mathop{\rm coker}\nolimits}
\sloppy
\newtheorem{defn0}{Definition}[section]
\newtheorem{prop0}[defn0]{Proposition}
\newtheorem{conj0}[defn0]{Conjecture}
\newtheorem{thm0}[defn0]{Theorem}
\newtheorem{lem0}[defn0]{Lemma}
\newtheorem{corollary0}[defn0]{Corollary}
\newtheorem{example0}[defn0]{Example}

\newenvironment{defn}{\begin{defn0}}{\end{defn0}}
\newenvironment{prop}{\begin{prop0}}{\end{prop0}}
\newenvironment{conj}{\begin{conj0}}{\end{conj0}}
\newenvironment{thm}{\begin{thm0}}{\end{thm0}}
\newenvironment{lem}{\begin{lem0}}{\end{lem0}}
\newenvironment{cor}{\begin{corollary0}}{\end{corollary0}}
\newenvironment{exm}{\begin{example0}\rm}{\end{example0}}

% To adjust space between rows of arrays:
\newcommand{\msp}{\renewcommand{\arraystretch}{.5}}
\newcommand{\rsp}{\renewcommand{\arraystretch}{1}}

% and to create a short pmatrix:
\newenvironment{lmatrix}{\renewcommand{\arraystretch}{.5}\small
 \begin{pmatrix}} {\end{pmatrix}\renewcommand{\arraystretch}{1}}
\newenvironment{llmatrix}{\renewcommand{\arraystretch}{.5}\scriptsize
 \begin{pmatrix}} {\end{pmatrix}\renewcommand{\arraystretch}{1}}
\newenvironment{larray}{\renewcommand{\arraystretch}{.5}\begin{array}}
 {\end{array}\renewcommand{\arraystretch}{1}}

% this defines an arrow that can be made as long as desired
\def \a{{\mathrel{\smash-}}{\mathrel{\mkern-8mu}}
{\mathrel{\smash-}}{\mathrel{\mkern-8mu}} {\mathrel{\smash-}}{\mathrel{\mkern-8mu}}}

\title[Euler characteristic of coherent sheaves on simplicial torics]%
{Euler characteristic of coherent sheaves on \\simplicial torics via the Stanley-Reisner ring}
\author{Hal Schenck}
\thanks{Schenck supported by NSF 07--07667, NSA 904-03-1-0006}
\address{Schenck: Mathematics Department \\ University of
 Illinois \\
   Urbana \\ IL 61801\\USA}
\email{schenck@math.uiuc.edu}

\subjclass[2000]{14M25, 14C40} \keywords{Toric variety, Cox ring, Chow ring, cohomology}

\begin{abstract}
\noindent We combine work of Cox on the total coordinate ring of
a toric variety and results of Eisenbud-Musta\c{t}\v{a}-Stillman and 
Musta\c{t}\v{a} on cohomology of toric and monomial ideals to obtain a 
formula for computing $\chi(\mathcal{O}_X(D))$ for a Weil divisor $D$ on a 
complete simplicial toric variety $X_{\Sigma}$. The main point is to 
use Alexander duality to pass from the toric irrelevant ideal, which 
appears in the computation of $\chi(\mathcal{O}_X(D))$,
to the Stanley-Reisner ideal of $\Sigma$, which is used in defining 
the Chow ring of $X_{\Sigma}$.
\end{abstract}
\maketitle
 
\section{Introduction}\label{sec:one}
For a divisor $D$ on a smooth complete variety $X$, the
Hirzebruch-Riemann-Roch theorem describes the Euler characteristic
of $\mathcal{O}_X(D)$ in terms of intersection theory:
\[
\chi(\mathcal{O}_X(D)) = \int ch(D)\cdot Td(X).
\]
The divisor $D$ corresponds to a class $[D]$ in the Chow ring of $X$,
and $ch(D)$ consists of the first $n=dim(X)$ terms of the formal
Taylor expansion of $e^D$. The Todd class of $D$ is defined similarly,
but using the Taylor expansion for $\frac{D}{1-e^{-D}}.$
To define the Todd class of $X$, filter the tangent bundle 
$\mathcal{T}_X$ by line bundles $\mathcal{O}(D_i)$. Then 
one shows that $Td(X) = \prod_{i=1}^n Td(D_i)$ 
is independent of the filtration. 

Let $\Sigma \subseteq \mathbb{R}^n$ be a complete simplicial 
rational polyhedral fan with $d=|\Sigma(1)|$ rays, 
$X_{\Sigma}$ the associated toric 
variety, and $D \in \Cl(X_{\Sigma})$ a Weil divisor on $X_{\Sigma}$. 
We combine Alexander duality and the Cox ring with results of 
Musta\c{t}\v{a} \cite{M1} on monomial ideals to obtain a 
formula for the Euler characteristic of the associated rank one reflexive sheaf 
${\mathcal O}_{X_{\Sigma}}(D)$. 
%\vskip .04in
%\noindent {\bf Theorem}
Put $Z =\{0,1\}^d$ and ${\bf 1} =\{1\}^d$. Then for $l \gg 0$, 
\begin{equation}\label{MAIN}
\chi({\mathcal O}_X(D)) = \sum\limits_{{\bf m} \in Z \setminus {\bf 0}}(-1)^{|{\bf m}|-d+n} \dim_{\C}(S/I_\Sigma)_{{\bf 1}-{\bf m}}\cdot \dim_{\C}S_{l \cdot \phi({\bf m})+D}.
\end{equation}
Here $I_\Sigma$ denotes the Stanley-Reisner ideal, and 
$\Z^d \stackrel{\phi}{\rightarrow} \Cl(X_{\Sigma})$ is the standard
surjection of $\Z^d$ onto the class group. The Cox ring $S$ 
is a polynomial ring, graded by $\Cl(X_{\Sigma})$; on $S/I_{\Sigma}$ we
use the $\mathbb{Z}^d$ grading. We recall the definitions of these objects
in \S 2. Any coherent sheaf on a nondegenerate toric variety corresponds to a
finitely generated $\Cl(X_{\Sigma})$-graded $S$--module
(see \cite{C} for the simplicial case, and \cite{M2} for the general case), 
so such a sheaf has a resolution by rank one reflexive sheaves, and 
Equation~\ref{MAIN} yields a formula for $\chi({\mathcal F})$ for
any coherent sheaf ${\mathcal F}$. Bounds on $l$ are determined by 
Eisenbud-Musta\c{t}\v{a}-Stillman in \cite{EMS}, and are discussed
in \S 2.
\subsection*{Connections to physics and some history}
The methods which are used to prove Equation~\ref{MAIN} 
have applications to computations arising in mathematical 
physics: in a recent preprint \cite{BJRR}, 
Blumenhagen, Jurke, Rahn and Roschy conjectured an algorithm for 
computing the cohomology of line bundles on a toric variety. 
Their motivation was to compute massless modes in Type IIB/F
and heterotic compactifications, on a complete intersection in
a toric variety.  A strong form of the algorithm is established 
by Maclagan and Smith in Corollary~3.4 of \cite{MacS}; later proofs 
appear in Jow \cite{Jow} and Rahn-Roschy \cite{RR}. In all these
papers Alexander duality and results of \cite{EMS} play a key role, 
as they do in the proof of Equation~\ref{MAIN}. 
The original motivation for this work was to find a toric
proof for the Hirzebruch-Riemann-Roch theorem.
%it would be interesting to find
%such a proof. 
%For a smooth complete fan $\Sigma$, an inductive
%proof of Hirzebruch-Riemann-Roch using a result of Ishida \cite{ishida} 
%appears in \cite{CLS}. 
% gave proofs of the algorithm using 
%\cite{EMS} and Alexander duality.
%In \cite{sch}, the author suggested that their algorithm should 
%follow from results of Eisenbud-Musta\c{t}\v{a}-Stillman \cite{EMS} 
%and Musta\c{t}\v{a} \cite{M1}, combined with Alexander duality, and 
%included a proof of Equation~\ref{MAIN}. Some months later, Jow \cite{Jow}
%and Rahn-Roschy \cite{RR} gave proofs of the algorithm using 
%\cite{EMS} and Alexander duality.

The first toric interpretation of Hirzebruch-Riemann-Roch
is due to Khovanskii \cite{K}. In \cite{PK1}, \cite{PK2}, 
Pukhlikov-Khovanskii study additive measures on virtual polyhedra,
and obtain a Riemann-Roch formula for integrating
sums of quasipolynomials on virtual polytopes. 
Pommersheim \cite{PO1} and Pommersheim and
Thomas \cite{PO2} obtain results on Todd classes of simplicial 
torics, and in \cite{BV}, Brion-Vergne prove 
an equivariant Riemann-Roch for simplicial torics. 

The results of Eisenbud-Musta\c{t}\v{a}-Stillman in 
\cite{EMS} show that in the toric setting, 
$\chi(\mathcal{O}_X(D))$ may be
calculated via certain $Ext$ modules over the Cox ring of $X$.
On the other hand, evaluating the expression 
$\int ch(D)\cdot Td(X)$ involves a computation in the Chow
ring of $X$, and the Cox and Chow rings of a simplicial toric 
variety are connected by Alexander duality. 

The paper is structured as follows: in \S 2
we recall the results of \cite{EMS} and the computation of 
cohomology via the Cox ring. In \S 3 we introduce the 
Chow ring, recall that the Stanley-Reisner
ideal of $\Sigma$ is the Alexander dual of the toric irrelevant 
ideal of $\Sigma$, and use results of Musta\c{t}\v{a} and Stanley
to connect the parts. Equation~\ref{MAIN} is proved in \S 4, 
and illustrated on the Hirzebruch surface ${\mathcal H}_2$.
\subsection*{Toric facts}
Let $N \simeq \Z^n$ be a lattice, with dual lattice $M$, and 
let $\Sigma \subseteq N\otimes_{\Z}\R \simeq \R^n$ be a 
complete simplicial rational polyhedral fan (henceforth, 
simply fan), with $\Sigma(i)$ denoting the set of 
$i$-dimensional faces of $\Sigma$, and let $X_\Sigma$ be the 
associated toric variety. A Weil divisor on $X_\Sigma$ is of the form
\[
D = \sum_{\rho \in \Sigma(1)}a_\rho D_\rho, \mbox{ with }a_\rho \in \Z.
\]
Let $d = |\Sigma(1)|$. The class group of $X_\Sigma$ has a presentation
\[
0 \longrightarrow M \stackrel{\psi}{\longrightarrow} \Z^d \stackrel{\phi}{\longrightarrow} \Cl(X_\Sigma) 
\longrightarrow 0,
\]
where $\psi$ is defined by 
\[
\chi^{m} \mapsto \sum_{\rho \in \Sigma(1)}\langle m, v_{\rho} \rangle D_\rho, \mbox{ where }v_{\rho}\mbox{ is a minimal lattice generator for }\rho.
\] 
In \cite{C}, Cox introduced the total coordinate ring (henceforth called the {\em Cox ring}) of $X_\Sigma$. This is a polynomial ring, graded by the class group $\Cl(X_\Sigma)$.
\begin{defn} 
\[
S = \C[x_\rho \mid \rho \in \Sigma(1)] = \bigoplus\limits_{\alpha \in \Cl(X_\Sigma)}S_{\alpha}.
\]
\end{defn}
The utility of this grading is that for $\alpha \simeq D \in \Cl(X_{\Sigma})$, 
$H^0(\mathcal{O}_X(D)) \simeq S_{\alpha}$. 
%Finally, recall that for a toric variety $X_\Sigma$, the Euler sequence
%\[
%0 \longrightarrow \Omega^1_X \longrightarrow 
%\bigoplus\limits_{\rho \in \Sigma(1)} \cO_X(-D_\rho) \longrightarrow \cO_X^{n-d}
% \longrightarrow 0.
%\]
%is exact, and thus 
%\[
%Td(X_\Sigma) = \prod\limits_{\rho \in \Sigma(1)} \frac{D_\rho}{1-e^{-D_\rho}}.
%\]
For more background on toric varieties, see \cite{CLS}, \cite{danilov}, or \cite{F}.
\section{Cohomology and the Cox ring}\label{sec:two}
The Cox ring has a distinguished ideal, the {\em toric irrelevant ideal} 
\[
B(\Sigma) = \langle x^{\hat \sigma} \mid \sigma \in \Sigma \rangle, \mbox{ where }x^{\hat \sigma}=  \prod_{\rho \not\in \sigma(1)}x_\rho.
\]
Note that $B(\Sigma)$ is generated by monomials corresponding to the complements
of the maximal faces of $\Sigma$. For an ideal 
$I = \langle f_1,\ldots,f_m\rangle$ let 
\[
I^{[l]} =  \langle f_1^l,\ldots,f_m^l \rangle.
\]
In \cite{EMS}, Eisenbud-Musta\c{t}\v{a}-Stillman show that for 
$D \in \Cl(X_\Sigma)$, $i \ge 1$ and $l \gg 0$, 
\begin{equation}\label{EMSmain}
H^i(\mathcal{O}_X(D)) \simeq Ext^{i+1}_S(S/B(\Sigma)^{[l]},S(D))_0,
\end{equation}
They also obtain a bound for $l$. Fix a basis for $M$, and let 
$A$ be a $d \times n$ matrix with a row for each 
ray $u_{\rho} \in \Sigma(1)$, written with respect to the fixed basis.
Define
\begin{equation}
\label{Bounds1}
\begin{aligned}
a &= \max( |\text{entries of $A$}|)\\
b &= \max( |\text{$(n-1)\times (n-1)$ minors of $A$}|)\\
c &= \min( |\text{nonzero $n\times n$ minors of $A$}|).
\end{aligned}
\end{equation}
Corollary~3.3 of \cite{EMS} shows that if $D = \sum_\rho a_\rho D_\rho$, then 
Equation~\ref{EMSmain} holds for
\begin{equation}\label{EMSbound}
l \ge n^2 \max_{\rho \in \Sigma(1)}(|a_\rho|) ab/c
\end{equation}
For brevity, we use lower case to denote $\dim_{\C}$ of an object, e.g. $s_{\alpha} = \dim_{\C}S_\alpha$. 
\begin{lem}\label{L1}
For $l \gg 0$ and $D \in \Cl(X_{\Sigma})$,
\begin{equation}\label{L22}
\chi(\mathcal{O}_X(D)) = \sum\limits_{i=0}^n (-1)^ih^i(D) = s_{D} - \sum\limits_{i=0}^{n+1} (-1)^i
ext^{i}_S(S/B(\Sigma)^{[l]},S(D))_0.
\end{equation}
\end{lem}
\begin{proof}
$Ext^{0}_S(S/B(\Sigma)^{[l]},S) = Ext^{1}_S(S/B(\Sigma)^{[l]},S) = 0$,
so this follows from \cite{EMS}.
\end{proof}
\begin{lem}\label{L2}
If $F_\bullet$ is a free resolution for $S/B(\Sigma)^{[l]}$, then
\begin{equation}\label{L23}
\begin{aligned}
\sum\limits_{i=0}^{n+1} (-1)^i ext^{i}_S(S/B(\Sigma)^{[l]},S(D))_0 &= \sum\limits_{i=0}^{d} (-1)^i \dim_{\C}F_i^{\vee}(D)_0\\
 &= \sum\limits_{i=0}^{d} (-1)^i \dim_{\C}(F_i)^{\vee}_D.
\end{aligned}
\end{equation}
\end{lem}
\begin{proof} Take Euler characteristics. \end{proof}
\begin{lem}\label{L3}
If $F_\bullet$ is a minimal free resolution for $S/B(\Sigma)^{[l]}$, then 
\[
\dim_{\C}(F_i)^{\vee}_D =  \sum\limits_{D' \in \Cl(X_\Sigma)}tor_i^S(S/B(\Sigma)^{[l]},\C)_{D'} \cdot s_{D'+D}.
\]
\end{lem}
\begin{proof}
Let $F_\bullet$ be a minimal free resolution for $S/B(\Sigma)^{[l]}$, and 
\[
r_i(D') = tor_i^S(S/B(\Sigma)^{[l]},\C)_{D'}.
\]
Then 
\[
F_i = \bigoplus\limits_{D' \in \Cl(X_\Sigma)} S(-D')^{r_i(D')}.
\]
Now dualize and take the shift by $D$ into account.
\end{proof}
\section{Combinatorial commutative algebra}\label{sec:three}

\subsection*{Taylor resolution}
We now observe that the multigraded betti numbers $r_i(D')$ of
$S/B(\Sigma)^{[l]}$ can be replaced with related numbers which 
arise from a Taylor resolution for $S/B(\Sigma)$.
The Taylor resolution \cite{T} of a monomial ideal is a variant
of the Koszul complex, which takes into account the LCM's of the monomials
involved. 

Let $I = \langle m_1, \ldots,m_k \rangle$ be a monomial ideal,
and consider a complete simplex with vertices labelled by the $m_i$, 
and each $n$-face $F$ labelled with the LCM of the $n+1$ monomials corresponding
to vertices of $F$. Define a chain complex where the differential on an $n$-face 
$F = [v_{i_0},\ldots, v_{i_n}]$ is
\[
d(F) = \sum\limits_{j = 0}^n (-1)^{j}\frac{m_F}{m_{F\setminus v_{i_j}}} F\setminus v_{i_j},
\]
with $m_{F}$ denoting the monomial labelling face $F$. 
As shown by Taylor, this complex is actually a resolution 
(though often nonminimal) of $I$. 
When the $m_i$ are squarefree, the LCM of a subset of $l^{th}$ powers is the 
$l^{th}$ power of the LCM of the original monomials, hence the Taylor 
resolution for $I^{[l]}$ is given by the $l^{th}$ power of the Taylor
resolution for $I$, in the sense that a summand $S(-\alpha)$ in the free
resolution for $I$ is replaced with $S(-l\cdot \alpha)$ in the resolution 
for $I^{[l]}$. 

Thus, the Taylor resolution of $S/B(\Sigma)$ determines
the Taylor resolution of $S/B(\Sigma)^{[l]}$. The formula in Lemma~\ref{L3}
requires a minimal free resolution, which the Taylor resolution is
generally not. However, this is no obstacle:
\begin{lem}\label{L4}
If $F_\bullet$ is a free resolution for $S/B(\Sigma)$, then
\[
\sum\limits_{i=0}^{n+1} (-1)^i
ext^{i}_S(S/B(\Sigma)^{[l]},S(D))_0 =  \sum\limits_{i=0}^{d} (-1)^i  \sum\limits_{D' \in \Cl(X_\Sigma)}tor_i^S(S/B(\Sigma),\C)_{D'} \cdot s_{l \cdot D'+D}.
\]
\end{lem}
\begin{proof}
If $F_\bullet$ is a minimal resolution of  $S/B(\Sigma)^{[l]}$, then Lemmas~\ref{L2} and \ref{L3} yield  
 \[
\sum\limits_{i=0}^{n+1} (-1)^i
ext^{i}_S(S/B(\Sigma)^{[l]},S(D))_0 = \sum\limits_{i=0}^{d} (-1)^i  \sum\limits_{D' \in \Cl(X_\Sigma)}tor_i^S(S/B(\Sigma)^{[l]},\C)_{D'} \cdot s_{D'+D}.
\]
Lemma~\ref{L2} shows that the $l^{th}$ power of a Taylor 
resolution for $S/B(\Sigma)$ can be used to compute the left-hand side. 
Furthermore, when $F_\bullet$ is non-minimal, in the
expression 
\[
\sum\limits_{i=0}^{d} (-1)^i \dim_{\C}(F_i)^{\vee}_D
\]
the nonminimal summands cancel out, hence we may pass back to the description in terms of $\Tor$, yielding
the result. 
%Remark: we could also use a hull resolution \cite{MS} to show this.
\end{proof}
\subsection*{Alexander duality and monomial ideals}
Let $\Delta$ be a simplicial complex on vertex set $\{1,\ldots,d
\}$. Let $S= \Z[x_1,\ldots,x_d]$ be a polynomial ring, with variables
corresponding to the vertices of $\Delta$. 
\begin{defn}\label{SR}
The Stanley-Reisner ideal $I_\Delta \subseteq S$ is the ideal 
generated by all monomials corresponding to nonfaces of $\Delta$:
\[
I_\Delta = \langle x_{i_1}\cdots x_{i_k} | [i_1,\ldots,i_k] \mbox{ is not a face of }\Delta \rangle.
\]
\end{defn}
The Stanley-Reisner ring is $S/I_\Delta$. 
The intersection of a complete simplicial fan $\Sigma \subseteq \R^n$ 
with the unit sphere $S^{n-1}$ gives a
simplicial complex we denote by $P_{\Sigma}$; define 
$I_\Sigma$ as the Stanley-Reisner ideal of $P_{\Sigma}$.
\begin{defn}
If $\Delta$ is a simplicial complex on $[d]= \{1,\ldots d\}$, then
the Alexander dual $\Delta^{\vee}$ is a simplicial complex consisting
of the complements of the nonfaces of $\Delta$:
\[
\Delta^{\vee} = \{ [d] \setminus \sigma | \sigma \not \in \Delta \}.
\]
\end{defn}
\begin{exm}\label{FirstExample}
The Hirzebruch surface ${\mathcal H}_2$ corresponds to the 
fan in Figure~1.
\begin{figure}
\begin{center}
\includegraphics[height=1.4in]{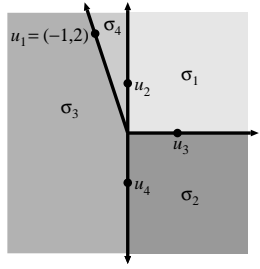}
\end{center}
\caption{The fan for ${\mathcal H}_2$} 
\end{figure}
Since  $[u_2,u_4]$ and $[u_1,u_3]$ are nonfaces of $\Sigma$, and every
other nonface such as $[u_1,u_2,u_4]$ contains them, the Stanley-Reisner
ideal is
\[
I_\Sigma = \langle x_1x_3, x_2x_4 \rangle.
\]
The Alexander dual $\Sigma^{\vee}$ contains all $\rho \in \Sigma(1)$.
Since $\widehat{u_1u_3} = [u_2,u_4]$ and $\widehat{u_2u_4} = [u_1,u_3]$, 
$\Sigma^{\vee}(2) = \{[u_2,u_4], [u_1,u_3]\}$. So 
\[
I_{\Sigma^{\vee}} = \langle x_1x_2, x_1x_4,x_2x_3,x_3x_4  \rangle.
\]
\end{exm}
\begin{lem}\label{L5}
The toric irrelevant ideal $B(\Sigma)$ is Alexander dual to the
Stanley-Reisner ideal $I_\Sigma$.
\end{lem}
\begin{proof}
The Alexander dual $I_{\Sigma^{\vee}}$ to $I_{\Sigma}$ is obtained by monomializing
(\cite{MS}, Proposition 1.35) a primary decomposition for $I_{\Sigma}$. If $MC(\Sigma)$ denotes the set of minimal cofaces of $\Sigma$, then the primary
decomposition of $I_\Sigma$ is 
\[
 I_{\Sigma} = \bigcap\limits_{[{i_1},\ldots,{i_k}] \in MC(\Sigma)}\langle x_{i_1},\ldots,x_{i_k}\rangle.
\]
The ideal $I_{\Sigma^{\vee}}$ is generated by monomials corresponding to minimal cofaces, which are complements to maximal faces, hence $I_{\Sigma^{\vee}} = B(\Sigma)$.
\end{proof}
\begin{thm}[Danilov \cite{danilov}, Jurkiewicz \cite{jurk}]
For a complete simplicial fan $\Sigma$, let 
$J = \langle div(\chi^{\bf m}) | {\bf m} \in M \rangle$. The 
rational Chow ring $Ch(X_\Sigma)$ is the rational 
Stanley-Reisner ring of $\Sigma$, modulo $J$. 
\end{thm}
The ideal $J$ is minimally generated by a regular sequence; it is these
linear forms which encode the geometry of $\Sigma$. 
To interpret the Euler characteristic of $\mathcal{O}_X(D)$ in 
terms of intersection
theory, we must change computations involving
the toric irrelevant ideal into computations involving the Stanley-Reisner
ideal. For a polynomial ring $R=\C[x_1,\ldots,x_d]$ endowed with 
the fine (also called $\Z^d$) grading $\deg(x_i) = {\bf e}_i \in \Z^d$ and 
squarefree monomial ideal $M$, the following result of 
Musta\c{t}\v{a} [\cite{M1}, Corollary 3.1] provides the bridge:
\begin{equation}\label{MirceaMonomial}
Tor_i^R(M^\vee, \C)_{{\bf m}} \simeq Ext_R^{|{\bf m}|-i}(R/M,R)_{-{\bf m}} \mbox{ if }{\bf m} \in \{0,1\}^d, \mbox{ else }0.
\end{equation}
Letting $Z =\{0,1\}^d$, applying 
Musta\c{t}\v{a}'s result yields:
\begin{equation}\label{MirceaMonomial2}
\begin{aligned}
tor_i^S(S/B(\Sigma),\C)_{D'} &= \sum\limits_{\stackrel{{\bf m} \in Z,}{ \phi({\bf m}) = D'}}tor_i^S(S/B(\Sigma),\C)_{{\bf m}}\\
 &= \sum\limits_{\stackrel{{\bf m} \in Z,}{\phi({\bf m}) = D'}}ext^{|{\bf m}|-i+1}_S(S/I_{\Sigma},S)_{-{\bf m}}
\end{aligned}
\end{equation}
\begin{lem}\label{L6}For a complete fan $\Sigma \subseteq N\otimes_{\Z}\R \simeq \R^n$ with $|\Sigma(1)|=d$,
\begin{enumerate}
\item
$Ext^j_S(S/I_{\Sigma},S) = 0 \mbox{ for all }j \ne d-n$.
\item In the $\Z^d$ grading, $Ext^{d-n}_S(S/I_{\Sigma},S) \simeq S/I_{\Sigma}({\bf 1})$.
\end{enumerate}
\end{lem}
\begin{proof}
From Definition~\ref{SR}, $I_{\Sigma}$ is the Stanley-Reisner ideal of 
the simplicial sphere $P_{\Sigma}$, which is Gorenstein by 
Corollary II.5.2 of \cite{S}. Since $\dim P_{\Sigma} = n-1$, 
\[
\codim(I_{\Sigma}) = (d-1)-(n-1) = d-n.
\]
Everything follows from this, save that $S/I_{\Sigma}$ is shifted by ${\bf 1}$. 
The Gorenstein property means the minimal free resolution of 
$S/I_{\Sigma}$ is of the form
%\begin{small}
\[
0 \longrightarrow S(-\alpha) \stackrel{\partial_{d-n}}{\longrightarrow}
\bigoplus\limits_{j=1}^kS(-\beta_j) \stackrel{\partial_{d-n-1}} {\longrightarrow} \cdots \longrightarrow \bigoplus\limits_{j=1}^kS(-\gamma_j) \stackrel{[I_{\Sigma}]} {\longrightarrow} S \longrightarrow S/I_{\Sigma} \longrightarrow 0,
\]
where $\partial_{d-n}$ is (up to signs) 
the transpose of the matrix of minimal generators
$[I_{\Sigma}]$. To show that the shift in
$Ext^{d-n}$ is ${\bf 1}$, we use a result of Hochster. 
For a complex $\Delta$ and weight $\alpha$, 
let $\Delta|_{\alpha} = \{\sigma \in \Delta \mid \sigma \subseteq \alpha \}$. 
Equating the multidegree ${\bf 1}$ with the full simplex on all vertices of
$\Delta$, Hochster's formula (5.12 of \cite{MS}) yields
\[
Tor_{d-n}^S(S/I_\Sigma,\C)_{\bf{1}} = \widetilde{H}^{n-1}(\Sigma|_{\bf{1}},\C).
\]
Since $\Sigma|_{\bf{1}} \simeq P_{\Sigma} \simeq S^{n-1}$, the result follows.
\end{proof}
\begin{exm}\label{SecondExample}
The Stanley-Reisner ring for the fan $\Sigma$ of 
Example~\ref{FirstExample} has a $\Z^4$ graded minimal free
resolution
\begin{small}
\[
0 \longrightarrow S(-1,-1,-1,-1) \xrightarrow{\left[ \!
\begin{array}{c}
-x_2x_4\\
x_1x_3
\end{array}\! \right]} \!
\begin{array}{c}
S(-1,0,-1,0)\\
\oplus \\ 
S(0,-1,0,-1)
\end{array}\!
\xrightarrow{\left[ \!\begin{array}{cc}
x_1x_3 & x_2x_4
\end{array}\! \right]}
 S \longrightarrow S/I_\Sigma.
\]
\end{small}

\noindent Thus, $Ext^2(S/I_\Sigma,S) \simeq S(1,1,1,1)/I_\Sigma$.
The simplicial complex $P_\Sigma$ consists of vertices $[1],[2],[3],[4]$
and edges $[12],[23],[34],[41]$ and is homotopic to $S^1$. Since the 
multidegrees are all smaller than ${\bf 1}$ in the pointwise order, 
$\Sigma|_{\bf 1} = P_\Sigma$, so 
\[
\C = \widetilde{H}^1(S^1,\C) = \widetilde{H}^1(\Sigma|_{\bf{1}},\C) = Tor_{2}^S(S/I_\Sigma,\C)_{\bf{1}},
\]
showing the shift $\alpha$ in the last step of the free resolution of
$S/I_\Sigma$ is $S(-{\bf 1})$.
\end{exm}

\section{Proof of Equation~\ref{MAIN}}\label{sec:three}
\noindent We now prove Equation~\ref{MAIN}. By Equation~\ref{L22}, 
\[
\chi(\mathcal{O}_X(D)) = s_{D} - \sum\limits_{i=0}^{n+1} (-1)^i
ext^{i}_S(S/B(\Sigma)^{[l]},S(D))_0.
\]
Let $\gamma({\bf m})\!=\!s_{l \cdot \phi({\bf m})+D}$ and $E\!=\! \sum\limits_{i=0}^{n+1} (-1)^i ext^{i}_S(S/B(\Sigma)^{[l]},S(D))_0$. It suffices to show 
\[
E =s_D + \sum\limits_{{\bf m} \in Z \setminus {\bf 0}}(-1)^{|{\bf m}|-d+n+1} \dim_{\C}(S/I_\Sigma)_{{\bf 1}-{\bf m}}\cdot \gamma({\bf m}).
\]
First, observe that 
\begin{equation}\label{FirstBigEqn}
\begin{aligned}
E &= \sum\limits_{i=0}^{d} (-1)^i \!\!\sum\limits_{D' \in \Cl(X_\Sigma)}\Big(\sum\limits_{\stackrel{{\bf m} \in Z,}{ \phi({\bf m}) = D'}}tor_i^S(S/B(\Sigma),\C)_{{\bf m}}\Big)\cdot \gamma({\bf m}).\\
 &= \sum\limits_{i=0}^{d} (-1)^i\sum\limits_{{\bf m} \in Z}tor_i^S(S/B(\Sigma),\C)_{{\bf m}} \cdot \gamma({\bf m}).\\
 &= s_D+\sum\limits_{i=1}^{d} (-1)^i\sum\limits_{{\bf m} \in Z \setminus {\bf 0}}tor_i^S(S/B(\Sigma),\C)_{{\bf m}} \cdot \gamma({\bf m}).
\end{aligned}
\end{equation}
The first line follows from Lemma~\ref{L4}, the second line is simply a 
rearrangement, and the third line follows from the observation that
\[
s_D = tor_0^S(S/B(\Sigma),\C)_{{\bf 0}} \cdot \gamma({\bf 0}).
\]
For $i \ge 0$, 
\[
Tor^S_i(B(\Sigma),\C) \simeq Tor^S_{i+1}(S/B(\Sigma),\C),
\]
so using Equation~\ref{MirceaMonomial} we may rewrite the last line of
Equation~\ref{FirstBigEqn} as
\begin{equation}\label{AlmostDone}
s_D + \sum\limits_{i=0}^{d-1} (-1)^{i+1} \sum\limits_{{\bf m} \in Z \setminus {\bf 0}}ext^{|{\bf m}|-i}_S(S/I_{\Sigma},S)_{-{\bf m}} \cdot \gamma({\bf m}).
\end{equation}
By Lemma~\ref{L6}, $Ext^{|{\bf m}|-i}_S(S/I_{\Sigma},S)$ is nonzero iff $|{\bf m}|-i = d-n$, and 
\[
Ext^{d-n}_S(S/I_{\Sigma},S) \simeq S/I_{\Sigma}({\bf 1}).
\]
Since the only nonzero terms in Equation~\ref{AlmostDone} occur for
$i = |{\bf m}|-d+n$ we rewrite Equation~\ref{AlmostDone} as
\begin{equation}\label{SecondBigEqn}
\begin{aligned}
&=  s_D + \sum\limits_{{\bf m} \in Z \setminus {\bf 0}}(-1)^{|{\bf m}|-d+n+1}ext^{d-n}_S(S/I_{\Sigma},S)_{-{\bf m}} \cdot  \gamma({\bf m})\\
& =  s_D + \sum\limits_{{\bf m} \in Z \setminus {\bf 0}}(-1)^{|{\bf m}|-d+n+1} \dim_{\C}(S/I_\Sigma)_{{\bf 1}-{\bf m}}\cdot  \gamma({\bf m})
\end{aligned}
\end{equation}
This shows that 
\[
E =s_D + \sum\limits_{{\bf m} \in Z \setminus {\bf 0}}(-1)^{|{\bf m}|-d+n+1} \dim_{\C}(S/I_\Sigma)_{{\bf 1}-{\bf m}}\cdot \gamma({\bf m}),
\]
and Equation~\ref{MAIN} follows. $\Box$
\begin{exm}
Consider the divisor $D=3D_3-5D_4$ on the 
Hirzebruch surface ${\mathcal H}_2$ from Figure~1.
Since the support function for $D$ is not convex, $D$ is not nef.
Thus, computing $\chi(\cO_{{\mathcal H}_2}(D))$ 
involves more than a simple global section computation. 
A direct calculation with Riemann-Roch for surfaces shows
that 
\[
\chi(\cO_{{\mathcal H}_2}(D)) =4.
\]
Using the methods of \S9.4 of \cite{CLS}, it can be shown
that $h^0(D) = 0$, $h^1(D) = 2$, and $h^2(D)=6$. 
Now we illustrate how to apply Equation~\ref{MAIN}. Let
\[
\phi = \begin{pmatrix}
1 & -2 & 1 & 0 \\
0  & 1 & 0 & 1 \end{pmatrix}
\]
so that the Class group is given by
\[
\Z^4 \stackrel{\phi}{\longrightarrow} \Z^2 \simeq Cl({\mathcal H}_2) \longrightarrow 0.
\]
The Eisenbud-Musta\c{t}\v{a}-Stillman bound of Equation~\ref{EMSbound} is $l= 80$, but
a careful analysis (see Example~3.6 of \cite{EMS}) shows that in this case 
taking $l = 4$ is sufficient. Then for example with ${\bf m}=(0,1,0,1)$ we have
$\phi({\bf m}) = (-2,2)$ so since $D = (3,-5)$, 
\[
S_{4\cdot \phi({\bf m}) + D} = S_{(-5,3)} = H^0(\cO_{{\mathcal H}_2}(-5,3)),
\]
and the dimension of this space is two. However, 
\[
(S/I_{\Sigma})_{{\bf 1}-(0,1,0,1)} =(S/I_{\Sigma})_{(1,0,1,0)}=0,
\]
since $x_1x_3  \in I_{\Sigma}$. A check shows that all 
terms in the summation vanish, save when 
\[
{\bf m} \in \{(1,1,0,1), (0,1,1,1), (1,1,1,1) \}
\]
For the first two values, $\phi({\bf m}) = (-1,2)$, and we compute
\[
S_{4\cdot \phi({\bf m}) + D} = S_{(-1,3)} = H^0(\cO_{{\mathcal H}_2}(-1,3)),
\]
which has dimension twelve. Since ${\bf 1}-{\bf m}$ is either
$(0,0,1,0)$ or $(1,0,0,0)$, for these two values of ${\bf m}$, 
\[
\dim_{\C}(S/I_{\Sigma})_{{\bf 1}-{\bf m}} =1
\]
Since $|{\bf m}|-d+n =1$, these two weights 
contribute $(-1) \cdot 2 \cdot 12 = -24$ to the
Euler characteristic. For the remaining weight ${\bf m}=(1,1,1,1)$,
the Stanley-Reisner ring is one dimensional in degree ${\bf 1}-{\bf m} = (0,0,0,0)$,
and $\phi(1,1,1,1)=(0,2)$ and 
\[
S_{4\cdot \phi({\bf m}) + D} = S_{(3,3)} = H^0(\cO_{{\mathcal H}_2}(3,3)),
\]
which has dimension $28$. Since $|{\bf m}|-d+n =2$ the contribution is
positive, thus
\[
\chi(\cO_{{\mathcal H}_2}(3D_3-5D_4)) = -24+28 = 4.
\]
\end{exm}

\noindent{\bf Problem} As noted in the introduction, this work
began as an attempt to find a toric proof of 
Hirzebruch-Riemann-Roch using Equation~\ref{MAIN}; 
it would be interesting to find such a proof. 
A proof of Equation~\ref{MAIN} also follows from results of 
Maclagan-Smith \cite{MacS}, I thank Greg Smith for noting this.
\vskip .04in
\noindent{\bf Acknowledgements} Computations were performed using 
{\tt Macaulay2} \cite{Macaulay2} by Grayson and Stillman, 
and {\tt NormalToricVarieties} \cite{NTV} by Greg Smith.
\bibliographystyle{amsalpha}

\end{document}